\documentclass[12pt]{article}
\usepackage{amsmath, amsthm}\usepackage{enumerate}\usepackage{amssymb}\usepackage{bbold}
\usepackage{bbm}\usepackage{dsfont}\usepackage{color}\usepackage{xcolor}\usepackage[explicit]{titlesec}
\newtheorem{theorem}{Theorem}[subsection]
\newtheorem{proposition}[theorem]{Proposition}

\newtheorem{lemma}[theorem]{Lemma}
\newtheorem{example}[theorem]{Example}
\newtheorem{corollary}[theorem]{Corollary}

\newtheorem{definition}[theorem]{Definition}

\makeatletter
\renewcommand{\subsection}{\@startsection{subsection}{1}
{0pt}{3.25ex plus 1ex minus.2ex}{-1em}{\normalfont\normalsize\bf}}
\makeatother
\begin{document}

\title{{\bf Limitedly L-weakly compact operators}}
\maketitle
\author{\centering{{Safak Alpay$^{1}$, Svetlana Gorokhova $^{2}$, Eduard Emelyanov$^{3}$\\ 
\small $1$ Middle East Technical University\\ 
\small $2$ Uznyj matematiceskij institut VNC RAN\\
\small $3$ Sobolev Institute of Mathematics}

\abstract{We introduce new class of limitedly L-weakly compact operators from a Banach space
to a Banach lattice. This class is a proper subclass of the Bourgain-Diestel operators and it contains
properly the class of L-weakly compact operators. We give its efficient characterization 
in term of sequences, investigate the domination problem, and study the completeness 
of this class of operators.}

\vspace{5mm}
%$*$Corresponding Author}%%\abstract{.}\\ 
\noindent
{\bf Keywords:} Banach lattice, \text{\rm Lwc}-set, limited set, limitedly L-weakly compact operator.

\noindent
{\bf MSC2020:} {\normalsize 46A40, 46B42, 46B50, 47B65}

}}
\bigskip
\bigskip

%%%%%%%%%%%%%%%%%%%%
\section{Introduction and Preliminaries}
%%%%%%%%%%%%%%%%%%%%

The L-weakly compact operators were introduced by P. Meyer-Nieberg in the beginning of seventies in 
order to diversify the concept of weakly compact operators via imposing Banach lattice structure 
on the range of operators. Limited operators and the Bourgain-Diestel operators were introduced 
a decade later. Since then, these operators attract permanent attention and 
inspire many researchers. In the last decade further related classes were introduced 
and studied by many authors \cite{BLM1,EAS,BLM3,BLM2}. 
Applying the Meyer-Nieberg approach to the Bourgain-Diestel operators, 
we introduce new classes of limitedly and bi-limitedly \text{\rm L}-weakly
compact operators, give their sequential characterization, and study the domination 
problem in these classes of operators.

In the present note: vector spaces are real; operators are linear and bounded; the letters $X$ and $Y$ 
stand for Banach spaces, $E$ and $F$ for Banach lattices; $B_X$ denotes the closed unit ball of $X$;
$\text{\rm L}(X,Y)$ the space of all bounded operators from $X$ to $Y$; 
$E_+$ the positive cone of $E$;
$$
  \text{\rm sol}(A)=\bigcup\limits_{a\in A}[-|a|,|a|]
$$
the solid hull of $A\subseteq E$; 
and $E^a=\{x\in E: |x|\ge x_n\downarrow 0\Rightarrow\|x_n\|\to 0\}$ 
the \text{\rm o}-continuous part of $E$. 
For further unexplained terminology on Banach lattices used in this note, we refer the reader to \cite{AlBu,AEG2,Kus,Mey}.

%The note is organized as follows. 
In Section 2, we introduce limitedly and bi-limitedly \text{\rm L}-weakly
compact operators and give a sequential characterization of such operators.
Section 3 is devoted to the domination problem in the class of limitedly \text{\rm L}-weakly
compact operators, and to the completeness of the spaces of such operators.

\section{Main Definitions and Basic Results}

Here we collect crucial definitions and prove the main technical theorem of the present note.

\subsection{}
The following definition is due to P. Meyer-Nieberg \cite{Mey}.

\begin{definition}\label{LWC-subsets}
{\em A subset $A$ of $F$ is an \text{\rm Lwc}-{\em set} whenever
each disjoint sequence $(a_n)$ in $\text{\rm sol}(A)$ is norm-null.
A bounded operator $T:X\to F$ is an \text{\rm Lwc}-{\em operator}
($T\in\text{\rm Lwc}(X,F)$)  if $T(B_X)$ is an \text{\rm Lwc}-subset of $F$.}
\end{definition}
%\noindent
%P. Meyer-Nieberg proved that each \text{\rm Lwc}-set is relatively weakly compact.

\smallskip
\noindent
As war as we know, the first part of the following definition goes back to J. Diestel.

\begin{definition}\label{bi-limited set}
{\em
A bounded subset $A$ of $X$ is called:
\begin{enumerate}[a)]
\item 
{\em limited} if each \text{\rm w}$^\ast$-null sequence in $X'$ is 
uniformly null on $A$. 
\item 
\text{\rm bi}-{\em limited} if each \text{\rm w}$^\ast$-null sequence in $X'''$ is 
uniformly null on $A$. 
\end{enumerate}}
\end{definition}
\noindent
It is well known that $B_X$ is not limited in $X$ when $\dim(X)=\infty$, and 
$$
   A \ \text{\rm is relatively compact} \ \Longrightarrow A \ 
   \text{\rm is limited} \ \Longrightarrow A \ \text{\rm is bi-limited}. 
$$
Each \text{\rm bi}-limited subset of a reflexive Banach space is relatively compact.
$B_{c_0}$ is not limited, yet is \text{\rm bi}-limited in $c_0$ by Phillip's lemma. 
The following definition goes back to G. Emmanuele \cite{Emma}.

\begin{definition}\label{BD property} 
{\em A Banach space $X$ is said to possess the {\em Bourgain--Diestel property} 
(briefly, $X\in(\text{\rm BDP})$)
if each limited subset of $X$ is relatively weakly compact.
An operator $T:X\to Y$ is a {\em Bourgain--Diestel operator} 
(briefly, $T\in\text{\rm BD}(X,Y)$) if $T$ carries 
limited sets onto relatively weakly compact sets.}
\end{definition}
\noindent
It was shown in \cite{BD} that all separable and all reflexive Banach spaces belong to (\text{\rm BDP}).
%A Dedekind $\sigma$-complete Banach lattice $E$ belongs $(\text{\rm GPP})$ 
%iff $E$ has $\text{\rm o}$-continuous norm \cite{Buh}.
In particular, $\ell^1,c_0,c\in(\text{\rm BDP})$, 
yet $\ell^\infty\not\in(\text{\rm BDP})$ as $B_{c_0}$ is a limited but not weakly 
compact subset of $\ell^\infty$.  

\subsection{}
Here, we present two new classes of operators.
The first one lies strictly between \text{\rm L}-weakly compact and Bourgain--Diestel operators.
It refines the class of Bourgain--Diestel operators
in the same manner as \text{\rm L}-weakly compact operators refine weakly compact operators.
%Details are given in the diagrams below. 

\bigskip
\begin{definition}\label{Main LWC operators}
{\em An operator $T:X\to F$ is called:
\begin{enumerate}[i)]
\item 
{\em limitedly \text{\rm L}-weakly compact} ($T\in\text{\rm l-Lwc}(X,F)$), 
if $T$ carries limited subsets of $X$ onto \text{\rm Lwc}-subsets of $F$.
\item 
\text{\rm bi}-{\em limitedly \text{\rm L}-weakly compact} ($T\in\text{\rm bi-l-Lwc}(X,F)$), 
if $T$ carries \text{\rm bi}-limited subsets of $X$ onto \text{\rm Lwc}-subsets of $F$.
\end{enumerate}}
\end{definition}
\noindent
Clearly, $\text{\rm l-Lwc}(X,F)$ and $\text{\rm bi-l-Lwc}(X,F)$ are vector spaces and
$$\text{\rm Lwc}(X,F)\subseteq\text{\rm bi-l-Lwc}(X,F)\subseteq
\text{\rm l-Lwc}(X,F)\subseteq\text{\rm BD}(X,F).$$ 
All three inclusions here are proper in general.

\begin{example}\label{bi-l-Lwc operator that is not Lwc}
{\em 
Consider operator $I_{\ell^2}$. 
Since each limited subset of $\ell^2$ is relatively compact and relatively compact 
subsets of $\ell^2$ are \text{\rm Lwc}-sets, then 
$I_{\ell^2}\in\text{\rm l-Lwc}(\ell^2)$ and hence $I_{\ell^2}\in\text{\rm bi-l-Lwc}(\ell^2)$
because $\ell^2$ is reflexive. Since $B_{\ell^2}$ is not an \text{\rm Lwc}-set then 
$I_{\ell^2}\notin\text{\rm Lwc}(\ell^2)$.
}
\end{example}

\begin{example}\label{l-Lwc operator that is not bi-l-Lwc}
{\em 
Consider operator $I_{c_0}$. 
Since each limited subset of $c_0$ is relatively compact and relatively compact 
subsets of $c_0$ are \text{\rm Lwc}-sets, then 
$I_{c_0}\in\text{\rm l-Lwc}(c_0)$.
Since $B_{c_0}$ is a \text{\rm bi}-limited subset of $c_0$ which is not
an \text{\rm Lwc}-set then $I_{c_0}\notin\text{\rm bi-l-Lwc}(c_0)$.
}
\end{example}

\begin{example}\label{BD operator that is not l-Lwc}
{\em 
Consider $I_{c}$. Since $c$ is separable, limited subsets of $c$ coincide 
with relatively compact subsets, 
and hence $I_{c}\in\text{\rm BD}(c)$.
However, the singleton $\{(1,1,1,...)\}$ is limited
yet not an \text{\rm Lwc}-set in $c$. Thus 
$I_{c}\notin\text{\rm l-Lwc}(c)$.
}
\end{example}
\noindent
Combining the above examples we obtain: 
$$
   \text{\rm Lwc}(\ell^2\oplus c_0\oplus c)\subsetneqq
   \text{\rm bi-l-Lwc}(\ell^2\oplus c_0\oplus c)\subsetneqq
   \text{\rm l-Lwc}(\ell^2\oplus c_0\oplus c)\subsetneqq
   \text{\rm BD}(\ell^2\oplus c_0\oplus c).
$$

\subsection{}
The following sequential characterization of \text{\rm l-Lwc}-operators
restates the definition of \text{\rm l-Lwc}-operators in an applicable way, 
with no use of both limited and \text{\rm l-Lwc}-sets.

\begin{theorem}\label{l-LW-operators}
$T\in\text{\rm l-Lwc}(X,F)$ iff 
$T'f_n\stackrel{\text{\rm w}^\ast}{\to}0$ in $X'$ for each disjoint bounded
sequence $(f_n)$ in $F'$.
\end{theorem}
\noindent
The proof is based on the next important fact due to Meyer-Nieberg
(see \cite[Thm.5.63]{AlBu} and \cite[Prop.2.2]{BuDo} for a slightly more general setting).

\begin{proposition}\label{Burkinshaw--Dodds}
Let $A\subseteq E$ and $B\subseteq E'$ be nonempty bounded sets. Then
every disjoint sequence $(a_n)$ of $\text{\rm sol}(A)$ is uniformly null
on $B$ iff every disjoint sequence $(b_n)$ of $\text{\rm sol}(B)$ is uniformly null on $A$.
\end{proposition}

\begin{proof}(of Theorem~\ref{l-LW-operators}). 
($\Longrightarrow$) Let $T\in\text{\rm l-Lwc}(X,F)$, $(f_n)$ be disjoint bounded in $F'$,
and $x\in X$. As $\{Tx\}\subset F$ is an $\text{\rm Lwc}$-set, then
$T'f_n(x)=f_n(Tx)\to 0$ by Proposition~\ref{Burkinshaw--Dodds}. Since $x\in X$
is arbitrary, $(T'f_n)$ is $\text{\rm w}^\ast$-null.
($\Longleftarrow$) Let $T'f_n\stackrel{\text{\rm w}^\ast}{\to}0$ 
for each disjoint bounded $(f_n)$ in $F'$. Suppose $T\not\in\text{\rm l-Lwc}(X,F)$. 
So, there is a non-empty limited subset $L$ of $X$ such that $T(L)$ is not 
an $\text{\rm Lwc}$-set in $F$.
Then, by Proposition \ref{Burkinshaw--Dodds}, there exists a disjoint sequence
$(g_n)$ of $B_{F'}$ that is not uniformly null on $T(L)$.
Therefore, $(T'g_n)$ is not uniformly null on $L$ violating 
$T'g_n\stackrel{\text{\rm w}^\ast}{\to}0$ and the fact that $L$ is limited in $X$.
The obtained contradiction completes the proof.
\end{proof}

\section{Main Results}

In the present section we investigate: the domination problem 
in the class of \text{\rm l-Lwc}-operators; the completeness 
of the spaces of such operators in operator norm; and complete norms 
on the linear spans of such operators. We focus on \text{\rm l-Lwc}-operators,
as the results for \text{\rm bi-l-Lwc}-operators are basically similar.

\subsection{}
We begin with the following domination result. 

\begin{proposition}\label{l-LW-domination}
Let $0\le S\le T\in\text{\rm l-Lwc}(E,F)$ then $S\in\text{\rm l-Lwc}(E,F)$.
\end{proposition}
\begin{proof}
{\small Let $(f_n)$ be disjoint bounded in $F'$. Then $(|f_n|)$ is also disjoint bounded,
and hence $T'|f_n|\stackrel{\text{\rm w}^\ast}{\to}0$ by Theorem \ref{l-LW-operators}.
It follows from
$$
   |S'f_n(x)|\le S'|f_n|(|x|)\le T'|f_n|(|x|)\to 0  \ \ \ \ (\forall x\in E)
$$
that $S'f_n\stackrel{\text{\rm w}^\ast}{\to}0$. Using Theorem \ref{l-LW-operators} again, we conclude $S\in\text{\rm l-Lwc}(E,F)$.}
\end{proof}

\subsection{}
The following useful fact is another application of Theorem \ref{l-LW-operators}. 

\begin{proposition}\label{l-LW-closed}
Let $\text{\rm l-Lwc}(X,F)\ni T_n\stackrel{\|\cdot\|}{\to} T$. 
Then $T\in\text{\rm l-Lwc}(X,F)$.
\end{proposition}
\begin{proof}
{\small Let $(f_n)$ be disjoint bounded in $F'$, and $x\in X$. By Theorem \ref{l-LW-operators},
we need to show $T'f_n(x)\to 0$. Let $\varepsilon>0$. Pick any $k\in\mathbb{N}$
with $\|T-T_k\|\le\varepsilon$. Since $T_k\in\text{\rm l-Lwc}(X,F)$  
then $|T_k'f_n(x)|\le\varepsilon$ for $n\ge n_0$. As $\varepsilon>0$ is arbitrary,
it follows from
$$
   |T'f_n(x)|\le|T'f_n(x)-T'_kf_n(x)|+|T_k'f_n(x)|\le 
$$
$$
   \|T'-T'_k\|\|f_n\|\|x\|+|T_k'f_n(x)|\le(\|f_n\|\|x\|+1)\varepsilon 
   \ \ \ \ (\forall n\ge n_0)
$$
that $T'f_n(x)\to 0$, as desired.}
\end{proof}

\begin{corollary}\label{l-LW-algebra}
$\text{\rm l-Lwc}(E)$ is a closed right ideal in $\text{\rm L}(E)$ 
$($and hence a subalgebra of $\text{\rm L}(E)$$)$, and it is unital iff 
$I_E$ is \text{\rm l-Lwc}.
\end{corollary}

\begin{proof}
{\small By Proposition \ref{l-LW-closed},
$\text{\rm l-Lwc}(E)$ is a closed subspace of $\text{\rm L}(E)$. 
Since each bounded operator maps limited sets onto limited sets, $\text{\rm l-Lwc}(E)$ 
is a right ideal in $\text{\rm L}(E)$.
The condition on $I_E$ making $\text{\rm l-Lwc}(E)$ unital is trivial.}
\end{proof}

\subsection{}
Let $\emptyset\ne{\cal P}\subseteq\text{\rm L}(E,F)$ be a set of operators 
(the set of ${\cal P}$-{\it operators}). An operator $T:E\to F$ is 
%a {\it regularly} ${\cal P}$-{\it operator} (= 
an \text{\rm r}-${\cal P}$-{\it operator}, if $T=T_1 - T_2$ where  $T_1$ and 
$T_2$ are positive ${\cal P}$-operators.
${\cal P}$-operators are said to satisfy the {\em domination property} if
$0\le S\le T\in {\cal P} \ \Longrightarrow  \  S\in {\cal P}$. 
We say that  $T\in\text{\rm L}(E,F)$ is 
${\cal P}$-{\em dominated} if $\pm T\le U\in{\cal P}$. 

\begin{proposition}\label{prop elem}
Let ${\cal P}\subseteq\text{\rm L}(E,F)$, ${\cal P}\pm{\cal P}\subseteq{\cal P}\ne\emptyset$, 
and $T\in\text{\rm L}(E,F)$.\\
\smallskip
i) $T \ \text{is an {\rm r}-}{\cal P}\text{-operator} \Longleftrightarrow T \ \text{is a}\  {\cal P}\text{-dominated} \  {\cal P}\text{-operator.} $ \\
\smallskip
ii) Assume the modulus $|T|$ of $T$ exists in $\text{\rm L}(E,F)$ and suppose that ${\cal P}$-operators satisfy 
the domination property. Then \\
\smallskip
$T \ \text{is an {\rm r}-}{\cal P}\text{-operator}  \Longleftrightarrow |T| \in {\cal P}$.
\end{proposition}

\begin{proof}
i) Let $T=T_1-T_2$, where  $T_1,T_2\in {\cal P}$ are positive.
${\cal P}\pm{\cal P}\subseteq{\cal P}$ implies $T\in {\cal P}$ and
$U=T_1+T_2\in {\cal P}$. From $\pm T\le U$ obtain that $T$ is ${\cal P}$-dominated. \smallskip \\
\medskip
Now, let $T\in{\cal P}$ be a ${\cal P}$-dominated. 
Take $U\in {\cal P}$ such that $\pm T\le U$.
Since $T=U-(U-T)$, and both $U$ and $U-T$ are 
positive ${\cal P}$-operators, $T$ is an \text{\rm r}-${\cal P}$-operator.\smallskip \\
ii) First assume $|T|\in {\cal P}$. 
Since $T=T_+-T_-$, $0\le T_{\pm}\le|T|\in {\cal P}$, the domination property 
implies that $T_+$ and $T_-$ are positive ${\cal P}$-operators, and hence $T=T_+-T_-$ 
is an \text{\rm r}-${\cal P}$-operator. \smallskip\\
Now, assume $T$ is an \text{\rm r}-${\cal P}$-operator. 
Then there are positive $T_1,T_2\in{\cal P}$ 
satisfying $T=T_1-T_2$. Since $0\le T_+\le T_1$ and 
$0\le T_-\le T_2$, the domination property implies $T_+,T_-\in{\cal P}$. 
Hence $|T|=T_++T_-$ is likewise a ${\cal P}$-operator.
\end{proof}

\begin{proposition}\label{vect lat}
Let $F$ be Dedekind complete, and ${\cal P}$ 
a subspace in $\text{\rm L}(E,F)$,
satisfying the domination property. 
Then $\text{\rm r-}{\cal P}(E,F)$
is an order ideal in the Dedekind complete vector lattice
$\text{\rm L}_r(E,F)$.
\end{proposition}

\begin{proof}
Since $F$ is Dedekind complete, $\text{\rm L}_r(E,F)$
is a Dedekind complete vector lattice.
By Proposition~\ref{prop elem}~ii), 
$T\in\text{\rm r-}{\cal P}(E,F)\Longrightarrow |T|\in\text{\rm r-}{\cal P}(E,F)$,
and hence $\text{\rm r-}{\cal P}(E,F)$ is a vector sublattice of $\text{\rm L}_r(E,F)$.
Since ${\cal P}$ satisfies the domination property, $\text{\rm r-}{\cal P}(E,F)$
is an order ideal in $\text{\rm L}_r(E,F)$.
\end{proof}
\noindent
The next lemma follows from Proposition \ref{prop elem} ii)
due to Proposition \ref{l-LW-domination}.

\begin{lemma}\label{prop elem l-LW}
Let a linear operator $T:E\to F$ possess the modulus. Then
$T\in\text{\rm r-l-Lwc}(E,F)$ iff $|T|\in\text{\rm r-l-Lwc}(E,F)$.
\end{lemma}

\begin{theorem}
The following statements hold. \\
i) $\text{\rm r-l-Lwc}(E)$ is a subalgebra of $\text{\rm L}_r(E)$. Moreover, 
$$
   \text{\rm r-l-Lwc}(E)=\text{\rm L}_r(E) \ \Longleftrightarrow I_E\in\text{\rm l-Lwc}(E).
$$ 

ii) If $E$ is Dedekind complete then 
$\text{\rm r-l-Lwc}(E)$ is a closed order ideal of the Banach lattice $(\text{\rm L}_r(E), \ \|\cdot\|_r)$, 
where 
$\ \|x\|_r:=\|~|x|~\|$ is the regular norm on $\text{\rm L}_r(E)$.
\end{theorem}

\begin{proof}
i)\ \
It follows from Corollary \ref{l-LW-algebra}
that $\text{\rm r-l-Lwc}(E)$ is a right ideal and hence is 
a subalgebra of $\text{\rm L}_r(E)$. 
The condition on $I_E$ under that $\text{\rm r-l-Lwc}(E)=\text{\rm L}_r(E)$ is trivial.

ii)\ \ 
It follows from Lemma \ref{prop elem l-LW} that $\text{\rm r-l-LW}(E)$ is 
a Riesz subalgebra of $\text{\rm L}_r(E)$. Proposition \ref{vect lat} implies 
that $\text{\rm r-l-Lwc}(E)$ is an order ideal of $\text{\rm L}_r(E)$.
Let $S,T\in\text{\rm r-l-LW}(E)$ satisfy $|S|\le|T|$.
In order to show $\|S\|_r\le\|T\|_r$,
observe that $|S|,|T|\in\text{\rm l-LW}(E)$ by Lemma \ref{prop elem l-LW}. Then 
$
   \|S\|_r=\|~|S|~\|\le\|~|T|~\|=\|T\|_r.
$
Due to \cite[Lm.1]{Emel}, Lemma \ref{l-LW-closed} implies 
that $(\text{\rm r-l-Lwc}(E), \ \|\cdot\|_r)$ is a Banach space.
Clearly, the regular norm $\|\cdot\|_r$ is submultiplicative. Indeed,
let $S,T\in{\text{\rm L}}(E)$. Then $|ST|\le|S|\cdot|T|$, and hence 
$
   \|ST\|_r\le\|~|S||T|~\|\le\|~|S|~\|\cdot\|~|T|~\|=
   \|S\|_r\cdot\|T\|_r,
$
as desired.
\end{proof}

{\tiny 
%%%%%%%%%%%%%%%%%%%%
}
%\newpage

\end{document}